\documentclass[a4paper,12pt]{article}
\usepackage{array,amsfonts,amsmath,graphicx,mathrsfs,amssymb,amsthm,latexsym}
\usepackage[latin1]{inputenc}
\usepackage{color}
\newtheorem{thm}{Theorem}[section]
\newtheorem{lemma}[thm]{Lemma}

\def \cH {{\cal H}}

\def \cP {{\cal P}}

\begin{document}
\title{\vspace{-10ex} ~~ \\
Resolvable  h-sun designs}
\author {Mario Gionfriddo
\thanks{Supported  by PRIN and I.N.D.A.M (G.N.S.A.G.A.), Italy}\\
\small Dipartimento di Matematica e Informatica \\
\small Universit\`a di Catania \\
\small Catania\\
\small Italia\\
{\small \tt gionfriddo@dmi.unict.it}  \\
Giovanni Lo Faro
\thanks{Supported  by PRIN, PRA and I.N.D.A.M (G.N.S.A.G.A.), Italy}\\
\small  Dipartimento di Matematica e Informatica \\
\small  Universit\`a di Messina \\
\small  Messina\\
\small Italia\\
{\small \tt lofaro@unime.it}\\
 Salvatore Milici
\thanks{Supported by MIUR and by C. N. R. (G. N. S. A. G. A.), Italy}\\
\small Dipartimento di Matematica e Informatica \\
\small Universit\`a di Catania \\
\small Catania\\
\small Italia\\
{\small \tt milici@dmi.unict.it}  \\
Antoinette Tripodi
\thanks{Supported  by PRIN, PRA and I.N.D.A.M (G.N.S.A.G.A.), Italy}\\
\small  Dipartimento di Matematica e Informatica \\
\small  Universit\`a di Messina \\
\small  Messina\\
\small Italia\\
{\small \tt atripodi@unime.it}}
\date{ }
\maketitle

\begin{abstract}
In this article we completely determine the spectrum for 
uniformly resolvable decompositions of the complete graph $K_v$ into $r$
 1-factors and $s$ classes
containing only copies of $h$-suns.
\end{abstract}

\vbox{\small
\vspace{5 mm}
 }

\section{Introduction 
 }\label{introduzione}

Given a collection $\cH$ of graphs, an {\em $\cH$-decomposition\/}
of a graph $G$ is a decomposition of the edge set of $G$ into
subgraphs (called {\em blocks}) isomorphic to some element of $\cH$. Such a decomposition is
said to be {\em resolvable\/} if it is possible to partition the blocks
into  classes $\cP_i$
 (often referred to as {\em parallel classes\/})
 such that every vertex of $G$ appears
in exactly one block of each $\cP_i$; a
class is called {\em uniform} if every block of the class is
isomorphic to the same graph from $\cH$.  A resolvable $\cH$-decomposition of $G$ is sometimes also
referred to as an {\em $\cH$-fac\-torization\/} of $G$, and a class can
be called an {\em $\cH$-factor\/} of $G$. The case where
 $\cH=\{K_2\}$ (a single edge) is known as a
 {\em 1-factorization\/}; for $G=K_v$
it is well known to exist if and only if $v$ is
even. A single class of a  1-factorization, that is a pairing of
all vertices, is also known as a {\em 1-factor\/}
 or {\em perfect matching}.
 
Uniformly resolvable decompositions of $K_v$ have  been studied in \cite{DQS}, \cite{GM}, \cite{KMT}, \cite{LMT},
 \cite{M}, \cite{MT}, \cite{R}, \cite{S1} and \cite{S2}. Moreover  when $\cH=\{G_1, G_2\}$
the question of the existence of a uniformly resolvable decomposition of $ K_{v}$ into $r>0$ classes of
$G_1$ and $s>0$ classes of $G_2$ have been studied  in the case in which the
number\/ $s$ of\/ $G_2$-factors is maximum. Rees and Stinson \cite{RS} have solved the case $\cH=\{K_2, K_3\}$; Hoffman and   Schellenberg \cite{HS} the case $\cH=\{K_2, C_k\}$; Dinitz, Ling and Danziger \cite{DLD} the case $\cH=\{K_2, K_4\}$; K\"{u}\c{c}\"{u}k\c{c}\.{i}f\c{c}\.{i}, Milici and Tuza \cite{KMT} the case  $\cH=\{K_3, K_{1,3}\}$;  K\"{u}\c{c}\"{u}k\c{c}\.{i}f\c{c}\.{i}, Lo Faro, S. Milici and Tripodi \cite{KLMT} the case  $\cH=\{K_2, K_{1,3}\}$.

\subsection{Definitions and notation}

An $h$-sun ($h\geq 3$) is a graph with $2h$ vertices  $\{a_1,a_2,\ldots ,a_h, b_1,b_2,\ldots ,b_h\}$,
 consisting of an $h$-cycle $C_h= (a_1,a_2, \ldots, a_h)$ and a 1-factor $\{\{a_1,b_1\}, \{a_2,b_2\},\ldots, $ $\{a_h,b_h\}$; in what follows we will denote the $h$-sun by $S(C_h)$=$(a_1,a_2,\ldots ,a_h; b_1,$ $b_2,\ldots ,b_h)$ or $S(C_h)$. An $h$-sun is also called a
 crown graph \cite{H}. The spectrum problem for a $h$-sun system of order $v$ have been solved for $h=3,4,5,6,8$
 \cite{FJLS1,LG,LGW}. Moreover cyclic $h$-sun systems of order $v$ have been studied in  \cite{FJLS1,FJLS2,WL}.

Let $C_{m(n)}$ denote the graph with vertex set $ \bigcup_{i=1}^m
X^{i}$, with $|X^{i}|=n$ for $i=1,2,\ldots ,m$ and $X^{i}\cap
X^{j}=\emptyset$ for $i\neq j$, and  edge set $\{\{u,v\}: u\in
X^{i}, \, v\in X^{j}, \, | i-j|\equiv 1\pmod{m} \}$.


In this paper we study the existence of a uniformly resolvable
decomposition of $K_v$  having $r$  1-factors
 and $s$ classes containing only  $h$-suns;
we will use the notation $(K_2,S(C_h))$-URD$(v;r,s)$ for such
 a uniformly resolvable decomposition of $K_{v}$.
 Further, we will use the
notation $(K_2,S(C_h))$-URGDD$(r,s)$ of $C_{m(n)}$ to denote a
uniformly resolvable decomposition of $C_{m(n)}$ into $r$  1-factors and $s$ classes containing only $h$-suns.

\section{Necessary conditions}

In this section we will give necessary conditions for the existence
of a uniformly resolvable decomposition of $K_v$ into
$r$  1-factors and $s$ classes of
h-suns.

\begin{lemma}
\label{lemmaP0} If there exists a\/ $(K_2, S(C_h))$-URD$(v;r,s)$, $s>0$,  then\/ $v\equiv 0\pmod{2h}$ and\/ $s\equiv 0\pmod{2}$.
\end{lemma}

\begin{proof}
Assume that there exists a $(K_2, S(C_h))$-URD$(v;r,s)$, $s>0$. By resolvability it follows that 
$v\equiv 0\pmod {2h}$. Counting the edges of $K_v$  we obtain

$$\frac{rv}{2}+\frac{(2h)sv}{2h}=\frac{v(v-1)}{2}$$
 and hence
\begin{equation}
  r+2s=(v-1). \end{equation}

Denote by $R$ the set of $r$ 1-factors and by $S$ the set of $s$ parallel classes of $h$-suns. Since the
classes of $R$ are regular of degree $1$, we have that every vertex
$x$ of $K_v$ is incident with $r$ edges in $R$ and
$(v-1)-r$ edges in $S$. Assume that the vertex $x$
appears in $a$ classes with degree $3$ and in $b$ classes with degree 1
in $S$. Since

$$a+b=s \ \ \mbox{and}\ \
3a+b=v-1-r,$$
the equality (1) implies that
$$3a+b=2(a+b) \ \Rightarrow \ a=b$$
and hence $s=2a$. This completes the proof.
\end{proof}

Given $v\equiv 0\pmod{2h}$, $h\geq3$,  define $J(v)$ according to the following table:

\vspace{4 mm}

\begin{center}
\begin{tabular}{|c|c|}
\hline
  $v  $ &  {$J(v)$}
\\
\hline
$0 \pmod{4h}$ & $\{(3+4x, \frac{v-4}{2}-2x), x=0,1,\ldots,\frac{v-4}{4}\} $\\
$2h\pmod{4h}$, $h$ even, & $\{(3+4x, \frac{v-4}{2}-2x), x=0,1,\ldots,\frac{v-4}{4}\}$\\
$2h \pmod{4h}$, $h$ odd, & $\{(1+4x, \frac{v-2}{2}-2x), x=0,1,\ldots,\frac{v-2}{4}\}$\\

\hline
 \end{tabular}

\bigskip

Table 2: The set $J(v)$.

 \end{center}

Since   a\/ $(K_2, S(C_h))$-URD$(v;v-1,0)$ exists for every $v\equiv 0\pmod{2}$, we focus on $v\equiv 0\pmod{2h}$, $h\geq3$.

\vspace{4 mm}

\begin{lemma}
\label{lemmaP2} If there exists a\/ $(K_2, S(C_h))$-URD$(v;r,s)$  then\/  $(r,s)\in J(v)$.
\end{lemma}
\begin{proof}
Assume there exists a $(K_2, S(C_h))$-URD($v;r,s)$. Lemma \ref{lemmaP0} and
Equation (1) give $s\equiv 0\pmod{2}$ and $r\equiv (v-1)\pmod{4}$ \ and \
so
\begin{itemize}
   \item
if $v\equiv 0\pmod{4h}$, then  $r \equiv3\pmod{4}$,
  \item
if $v\equiv 2h\pmod{4h}$, $h$ even,  then $ r \equiv3\pmod{4}$,
 \item
if   $v\equiv 2h\pmod{4h}$, $h$ odd, then $r \equiv1\pmod{4}$.
\end{itemize}
Letting $r=a+4x$, $a=1$ or $3$, in  Equation (1), we obtain $2s=(v-1)-a-4x$;
since $s$ cannot be negative, and $x$ is an integer, the value of
$x$ has to be in the range as given in the definition of $J(v)$.
\end{proof}

Let now
$URD(v; K_2, S(C_h))$ := $\{(r,s)$ : $\exists$ $(K_2, S(C_h))$-URD$(v;r,s)\}$. In this paper we completely solve the spectrum problem for such systems, i.e., characterize the existence of uniformly resolvable
decompositions of $K_{v}$ into $r$  1-factors and $s$ classes of $h$-suns by proving the following result:\\

\noindent \textbf{Main Theorem.} { \em
 For every  $v\equiv 0\pmod{2h}$,   $URD(v;K_2, S(C_h))$=$J(v)$}.

\section{Small cases and basic lemmas}

\begin{lemma}
\label{lemmaD2} $URD(6;K_2, S(C_3))\supseteq J(v)$.
\end{lemma}
\begin{proof}
The case $(5,0)$ is trivial. For the case $(1,2)$, let $V(K_{6}$)=$\mathbb{Z}_{6}$ and the classes  listed below:\\
$\{(0,1,2;5,4,3)$ \}$, $\{(3,5,4;0,1,2)$\}$, $\{\{0,4\}, \{1,3\}, \{2,5\}\}$.
\end{proof}

\begin{lemma}
\label{lemmaD3} $URD(12;K_2, S(C_3))\supseteq J(v)$.
\end{lemma}

\begin{proof}
The case $(11,0)$ is trivial. For the remaining cases, 
let $V(K_{12}$)=$\mathbb{Z}_{12}$ and the classes  listed below:
\begin{itemize}
\item
$(3,4)$:\\
$\{$(0,4,8; 10,2,7)$, $(1,5,9; 11,3,6)$ \}$, $\{ $(2,6,10;  8,0,5)$, $(3,7,11;9,1,4)$ \}$,\\
$\{$(0,5,11; 9,2,6)$, $(1,4,10; 8,3,7)$ \}$, $\{$(2,7,9; 11,0,4)$, $(3,6,8; 10,1,5)$ \}$,\\
$\{\{0,1\}, \{2,3\}, \{4,5\}, \{6,7\}, \{8,9\}, \{10,11\}\}$, \\
$\{\{0,2\}, \{1,3\} ,\{4,7\}, \{5,6\}, \{8,11\}\}, \{9,10\}$,\\
$\{\{0,3\}, \{1,2\}, \{4,6\}, \{5,7\}, \{8,10\}, \{9,11\}\}$;
\item
$(7,2)$:\\
$\{$(0,4,8; 10,2,7)$, $(1,5,9; 11,3,6)$ \}$, $\{ $(2,6,10;  8,0,5)$, $(3,7,11;9,1,4)$ \}$,\\
$\{\{0,1\}, \{2,3\}, \{4,5\}, \{6,7\}, \{8,9\}, \{10,11\}\}$, \\
$\{\{0,2\}, \{1,3\} ,\{4,7\}, \{5,6\}, \{8,11\}\}, \{9,10\}$,\\
$\{\{0,3\}, \{1,2\}, \{4,6\}, \{5,7\}, \{8,10\}, \{9,11\}\}$,\\
$\{\{0,5\}, \{1,10\}, \{2,11\}, \{3,4\}, \{6,8\}, \{7,9\}\}$, \\
$\{\{0,7\}, \{1,8\}, \{2,5\}, \{3,10\}, \{4,9\}, \{6,11\}\}$, \\
$\{\{0,9\}, \{1,6\}, \{2,7\}, \{3,8\}, \{4,10\}, \{5,11\}\}$, \\
$\{\{0,11\}, \{1,4\}, \{2,9\}, \{3,6\}, \{5,8\}, \{7,10\}\}$.
\end{itemize}
\end{proof}

\begin{lemma}
\label{lemmaD9}  There exists a \/ $(K_2, S(C_h))$-URGDD$(r,s)$ of $C_{h(2)}$, for every $(r,s)\in \{(0,2),(4,0)\}$.

\end{lemma}

\begin{proof}

Consider  the sets $X^i=\{a_i, b_i\}$, for $ i=1,2,\ldots,h$,  and take the classes listed below (where we assume $h+1=1$):
\begin{itemize}
\item
$(0,2)$:\\
$\{(a_1,a_2,\ldots,a_h; $ $b_2,b_3,\ldots,b_h,b_1)\}$, $\{(b_1,b_2,\ldots,b_h; a_2,a_3,\ldots,a_h,a_1)\}$;
\item
$(4,0)$, $h$ even:\\
$\{\{a_{1+2i},a_{2+2i}\} , \{b_{1+2i},b_{2+2i}\}\, :\, i=0,1,\ldots, \frac{h}{2}-1\}$, \\
$\{\{a_{2+2i},a_{3+2i}\} , \{b_{2+2i},b_{3+2i}\}\, :\, i=0,1,\ldots, \frac{h}{2}-1\}$,\\
$\{\{a_{1+i},b_{2+i}\}\, :\, i=0,1,\ldots, h-1\}$, \\
$\{\{a_{2+i},b_{1+i}\}\, :\, i=0,1,\ldots, h-1\}$; 
\item
$(4,0)$, $h$ odd:\\
$\{\{a_{1+2i},a_{2+2i}\} , \{b_{2+2i},b_{3+2i}\}\, :\, i=0,1,\ldots, \frac{h-5}{2}\}\cup\{\{a_{h-2},a_{h-1}\},\{a_{h}, $ $b_{h-1}\},\{b_{1},b_{h}\}\}$, \\
$\{\{a_{2+2i},a_{3+2i}\} , \{b_{1+2i},b_{2+2i}\}\, :\, i=0,1,\ldots, \frac{h-3}{2}\}\cup\{\{a_{1},b_{h}\}\}$, \\
$\{\{a_{2+i},b_{1+i}\}\, :\, i=0,1,\ldots, h-3\}\cup\{\{a_{1}, a_{h}\},\{b_{h-1},b_{h}\}\}$, \\
$\{\{a_{1+i},b_{2+i}\}\, :\, i=0,1,\ldots, h-1\}$.
\end{itemize}
\end{proof}

\section{Main results}

\begin{lemma}
\label{lemmaC1} For every\/ $v\equiv 0\pmod{4h}$, $h\geq3$, $URD(v;K_2, S(C_h))\supseteq J(v)$.

\end{lemma}

\begin{proof}
Let $v=4ht$. The case $h=3$ and $t=1$ corresponds to a $(K_2, S(C_{3}))$-URD $(12;r,s)$ which follows by Lemma
\ref{lemmaD3}. For $t\geq2$, start with a $C_h$-factorization $\cP_1$, $\cP_2\ \ldots,$ $\cP_l$,  $l=ht-1$, of $K_{2ht}-F$ which comes from \cite{HS} and give weight 2 to each point of $X$.  Fixed any integer $0\leq x \leq l$,  for each $h$-cycle $C$ of $x$ parallel classes  place on $C\times \{1,2\}$ a copy of a $(K_2, S(C_h))$-URGDD$(4,0)$ of $C_{(h)2}$, while  for each $h$-cycle $C$ of the remaining classes place a copy of a $(K_2, S(C_h))$-URGDD$(0,2)$ of $C_{(h)2}$ (the input designs are from Lemma
\ref{lemmaD9}); for each edge $e\in F$ consider  a 1-factorization of  $K_4$ on $e\times \{1,2\}$. The result is a resolvable decomposition of $K_{v}$ into  $3+4x$ 1-factors and $\frac{v-4}{2}-2x$ classes of $h$-suns.

\end{proof}

\begin{lemma}
\label{lemmaC2} For every\/ $v\equiv 2h\pmod{4h}$, $h\geq3$ even, $URD(v;K_2, S(C_h))\supseteq J(v)$.

\end{lemma}

\begin{proof}
Let $v=2h+4ht$. Starting with  a $C_h$-factorization of $K_{h+2ht}-F$, which comes from \cite{HS}, the assertion follows by a similar argument as in  Lemma \ref{lemmaC1}.
\end{proof}

\begin{lemma}
\label{lemmaC3} For every\/ $v\equiv 2h\pmod{4h}$, $h\geq3$ odd, $URD(v;K_2, S(C_h))\supseteq J(v)$.

\end{lemma}

\begin{proof}
Let $v=2h+4ht$. The case $h=3$ and $t=0$ corresponds to a $(K_2, S(C_{3}))$-URD $(6;1,2)$ which follows by Lemma
\ref{lemmaD2}. For $t\geq 1$, start with a $C_h$-factorization $\cP_1$, $\cP_2\ \ldots,$ $\cP_l$,  $l=ht+\frac{h-1}{2}$,  of $K_{h+2ht}$ which comes from \cite{ASSW} and give weight 2 to each point of $X$.  Fixed an integer $0\leq x \leq l$,  for each $h$-cycle $C$ of $x$ parallel classes  place on $C\times \{1,2\}$ a copy of a $(K_2, S(C_h))$-URGDD$(4,0)$ of $C_{(h)2}$, while  for each $h$-cycle $C$ of the remaining classes place a copy of a $(K_2, S(C_h))$-URGDD$(0,2)$ of $C_{(h)2}$ (the input designs are from Lemma
\ref{lemmaD9}). If we consider also  the 1-factor  consisting of the edges $ \{x_1,x_2\}$ for $x\in X$, the result is a resolvable decomposition of $K_{v}$ into  $1+4x$ 1-factors and $\frac{v-2}{2}-2x$ classes of $h$-suns. 
\end{proof}

Combining Lemmas \ref{lemmaC1}, \ref{lemmaC2}, \ref{lemmaC3}, we obtain our main theorem.

\begin{thm}
For each\/  $v\equiv
0\pmod{2h}$, $URD(v;K_2, S(C_h))= J(v)$.

\end{thm}

\end{document}